\documentclass[]{elsarticle}%
\usepackage{lipsum}
\makeatletter
\def\ps@pprintTitle{%
 \let\@oddhead\@empty
 \let\@evenhead\@empty
 \def\@oddfoot{}%
 \let\@evenfoot\@oddfoot}
\makeatother

\usepackage{amsmath}
\usepackage[dvipsnames]{xcolor}

\usepackage{array}
\usepackage{amssymb}
\usepackage{amsthm}
\theoremstyle{plain}

\usepackage{dsfont} 

\usepackage[margin=1in]{geometry}
\usepackage[bookmarks=false]{hyperref}
\usepackage{epstopdf}

\def\xx{{\bf x}}

\def\+{{\oplus}}

\newcommand{\Ff}{{\mathbb F}}

\newcommand{\rme}{\mathrm{e}}
\newcommand{\rmd}{\mathrm{d}}
\newcommand{\rmi}{\mathrm{i}}

\newcounter{comments}
\newtheorem*{theorem*}{Theorem}
\newtheorem{thm}{Theorem}[section]
\newtheorem{theorem}[thm]{Theorem}
\newtheorem{lemma}[thm]{Lemma}
\newtheorem{cor}[thm]{Corollary}

\newtheorem{prop}[thm]{Proposition}

\newtheorem{defn}[thm]{Definition}

\newtheorem{remark}[thm]{Remark}

\numberwithin{equation}{section}

\begin{document}

\title{Multistate Nested Canalizing Functions and Their Networks}

\author[uzh,usz]{C. Kadelka\corref{cor1}}
\ead{kadelka.claus@virology.uzh.ch}

\author[ws]{Y. Li}
\ead{liyu@wssu.edu}

\author[bsse]{J. Kuipers}
\ead{jack.kuipers@bsse.ethz.ch}

\author[ws]{J.O. Adeyeye}
\ead{adeyeyej@wssu.edu}

\author[uc,jac]{R. Laubenbacher}
\ead{laubenbacher@uchc.edu}

\cortext[cor1]{Corresponding author}

\address[uzh]{Institute of Medical Virology, University of Zurich, 8006 Zurich, Switzerland}
\address[usz]{Division of Infectious Diseases and Hospital Epidemiology, University Hospital Zurich, 8091 Zurich, Switzerland}
\address[ws]{Department of Mathematics, Winston-Salem State University, NC 27110, USA}
\address[bsse]{D-BSSE, ETH Zurich, Mattenstrasse 26, 4058 Basel, Switzerland}
\address[uc]{Center for Quantitative Medicine, University of Connecticut Health Center, Farmington, CT 06030, USA}
\address[jac]{Jackson Laboratory for Genomic Medicine, Farmington, CT 06030, USA}


\begin{abstract}

This paper provides a collection of mathematical and computational tools for the study of robustness in nonlinear gene regulatory networks, represented by time- and state-discrete dynamical systems taking on multiple states.
The focus is on networks governed by nested canalizing functions (NCFs), first introduced in the Boolean context by S. Kauffman.
After giving a general definition of NCFs we analyze the class of such functions. We derive a formula for the normalized average $c$-sensitivities of multistate NCFs, which enables the calculation of the Derrida plot, a popular measure of network stability. We also provide a unique canonical parametrized polynomial form of NCFs. This form has several consequences. We can easily generate NCFs for varying parameter choices, and derive a closed form formula for the number of such functions in a given number of variables, as well as an asymptotic formula. Finally, we compute the number of equivalence classes of NCFs under permutation of variables.
Together, the results of the paper represent a useful mathematical framework for the study of NCFs and their dynamic networks.
\end{abstract}

\maketitle

\section{Introduction}
\label{sec-intro}

{
Many biological networks, in particular gene regulatory networks, are inherently stochastic in nature \cite{Elowitz02,Volfson06} and questions regarding the robustness of such networks have received much attention \cite{Strogatz, Boccaletti, Kitano}. The properties of molecular networks are frequently studied using time- and state-discrete models like Boolean and logical network models that have been used for this purpose since the 1970s \cite{Kau4, Thomas73}. Many of the more recently published discrete dynamical models include, however, variables that take on more than two states, needed to capture mechanisms that are not binary in nature; see, e.g., \cite{Albert, Brandon, Poltz, Thieffry}.

The impact of structural and topological network properties on the resulting dynamics is of particular interest when trying to understand the robustness of molecular networks, see, e.g., \cite{Li04,Fretter,Balleza}. One popular measure of network robustness is the Derrida value of a network \cite{Derrida1}. It assesses how perturbations, caused, for instance, by changes in the environmental conditions, propagate through the network.

Closely related to questions regarding robustness, Waddington had already in the 1940s developed the concept of canalization as a possible explanation of why the outcome of embryonal development leads to predictable phenotypes in the face of widely varying environmental conditions as well as frequent genetic mutations \cite{Wad}. Different phenotypes can be thought of as ``valleys" into which development is channeled by canalizing mechanisms, inferring protection from ubiquitous perturbations. Kauffman was the first to use a version of this concept in Boolean network modeling. He studied gene regulatory networks with canalizing Boolean functions~\cite{Kau74}, as well as the special subclass of so-called nested canalizing functions (NCFs)~\cite{Kau2}. Since then, these functions and the dynamics of the networks they govern have been extensively studied \cite{Kau3,Abd2,Karlsson07,Kochi12,Mur,Jansen13,Kadelka_Derrida16}. The discovery of some invariants of a canalizing function enabled a further categorization of the set of all canalizing Boolean functions, see, e.g., \cite{Yua1,Layne12,He16}. Due to the increased need to model update functions with multistate variables, the concept of nested canalization has been generalized to functions taking values from any finitie field \cite{Mur,Mur2}.
A \emph{canalizing} function possesses at least one input variable such that, if this variable takes on a certain ``canalizing'' value, then the output value is already determined, regardless of the values of the remaining input variables. If this variable takes on a non-canalizing value, and there is a second variable with this same property, and so on, then the function is \emph{nested canalizing}.

The number of Boolean (nested) canalizing functions are known~\cite{Abd2,Win2}, and for the multistate case there exists at least a recursive formula for the number of NCFs \cite{Mur2}; the probability that a random Boolean function is canalizing decreases rapidly as the number of inputs increases (less than $0.5$ for three inputs, and less than $0.01$ for five inputs). Nested canalization is more restrictive and the probabilities are even lower ($0.25$ for three inputs, and less than $10^{-5}$ for five inputs). Interestingly, an analysis of published Boolean models of molecular networks revealed that all $139$ investigated rules with at least three inputs are canalizing and even that $133$ are nested canalizing~\cite{Kau2, Mur, Harris}. (Unfortunately, a similar study has yet to be conducted for multistate models.) 
These findings clearly motivate the study of general multistate NCFs in the context of understanding the regulatory logic of gene networks. 
}

In this paper, we introduce a collection of computational tools that can aid in the construction and analysis of multistate discrete dynamical systems, with a focus on nested canalization and robustness. Specifically, the novel contributions of the paper are as follows:
\begin{itemize}
\item
Multistate, rather than Boolean, models of gene regulatory and signaling networks are increasingly used to model biological phenomena \cite{Albert, Brandon, Poltz, Thieffry}. This suggests that theoretical studies of these, such as robustness, should be carried out in the multistate context. (We are not aware of any arguments that show that Boolean networks are sufficiently general for this purpose.) Therefore, we treat the multistate case and provide general results that can be used in that context. 
\item 
The Derrida values of a network can be expressed as a weighted sum of the normalized average $c$-sensitivities of the canalizing update functions \cite{Kadelka_Derrida16}. We provide a formula for the normalized average $c$-sensitivities of multistate NCFs. This greatly simplifies the application of the Derrida plot for robustness analyses of multistate nested canalizing networks, which otherwise requires extensive simulations (difficult or infeasible for large networks).
\item
In \cite{Alan}, the authors showed that polynomial functions over a finite field serve as a rich mathematical framework for the modeling of discrete dynamical systems. Using this framework, we provide a canonical parametrized polynomial form for multistate NCFs. This canonical form can be used to easily generate such functions, a nontrivial task for functions with many variables; as many robustness investigations require a large number of randomly generated networks, this parametrization is particularly useful. 
\item
Using the polynomial framework, we present a closed formula for the number of multistate NCFs in a given number of variables, in terms of the Sterling numbers of the second kind, and derive an explicit exponential generating function for them. From this generating function, we can derive a simple asymptotic approximation of the number of NCFs. This allows to carry out asymptotic studies, which we also provide. Knowing the proportion of NCFs among all functions provides information about the richness of this class of functions for the purpose of capturing a variety of regulatory mechanisms. 
\item
We compute the number of equivalence classes of NCFs under permutation of variables. This question has received recent attention \cite{Reichhardt,Rocha}. 
\end{itemize}

Together, these results provide a detailed characterization of multistate NCFs, including explicit formulas that tell us how prevalent these functions are, and that simplify the use of the common Derrida value metric for robustness analyses of dynamic networks constructed from these functions. We conclude the paper with a discussion of avenues for future work as well as problems and limitations in generalizing the concept of nested canalization.


\section{The concept of canalization}\label{2} 
This paper investigates multistate functions with inputs and outputs in finite sets. 
Functions appearing in models of gene regulatory networks typically take values in sets of different cardinality, as variables can take on different numbers of states; for instance in logical models, many variables are Boolean, while some take on $3$ or more states \cite{Thomas91}. This complicates the general mathematical framework for such functions and their networks. It is explained in detail in \cite{Alan, Lau} that one can make the assumption that the functions we are considering take their inputs and produce their outputs in a finite field with a prime number of elements, rather than in a general finite set. This assumption does not limit the generality of the considered functions because the domain and range of any function can be enlarged, if needed, to a finite field with a prime number of elements by adding ``dummy" states to the domain. Thus, {\bf{throughout this paper, we will make the assumption that all functions are defined over a finite field $\Ff=\Ff_{p}$ with a prime number $p$ of elements.}} (Such a field can be thought of as the set of integers modulo $p$, with the usual addition and multiplication modulo $p$.) The Boolean case corresponds, of course, to $p=2$. An important consequence of this assumption, used heavily here, is that any function $f: \Ff\times\cdots\times\Ff\longrightarrow\Ff$ can be represented by a polynomial function with coefficients in $\Ff$ \cite{Alan}. 

In this section we review some concepts and definitions from \cite{Mur, Mur2} to introduce the computational concept of \emph{canalization}. 
\begin{defn} \label{def2.1} 
A function $f(x_{1},x_{2},\ldots, x_{n})$ is essential in the variable $x_{i}$ if there exist $r, s\in\Ff$ and $(x_1, \ldots, x_{i-1}, x_{i+1}, \ldots, x_n) \in \Ff^{n-1}$ such that
\[f(x_{1},\ldots,x_{i-1},r,x_{i+1},\ldots,x_{n})\neq f(x_{1},\ldots,x_{i-1},s,x_{i+1},\ldots,x_{n}).\]
\end{defn}

\begin{defn}\label{def2.2} 
A  function $f: \mathbb{F}^n \rightarrow \mathbb{F}$ is $\langle i:a:b\rangle$ canalizing if there exist a variable $x_i$, $a,b \in \mathbb{F}$ and a function $g(x_1,\ldots,x_n)$ such that
$$f(x_1,\ldots,x_n) = 
 \begin{cases}
  b & \ \text{if} \  x_i = a \\
  g\not\equiv b & \ \text{if} \ x_i \neq a,
  \end{cases}
$$
in which case $x_i$ is called a canalizing variable, the input $a$ is the canalizing input, and the output value $b$ when $x_i=a$ is the corresponding canalized output.
\end{defn}

%

We now assume that $\Ff=\{0,1,\ldots,p-1\}$ is ordered, in the natural order $0<1<\cdots <p-1$. 
A proper subset $S$ of $\Ff$ is called a segment if and only if $S=\{0,\ldots,j\}$ or $S^c=\Ff-S=\{0,\ldots,j\}$ for some $0\leq j<p-1$. Hence, a proper subset $S$ is a segment if and only if $S^c$ is a segment. 

\begin{defn}\label{def2.3}\cite{Mur2}
Let $f: \Ff^n \rightarrow \Ff$ be a  function in $n$ variables, let $\sigma \in \mathcal{S}_n$ be a permutation of $\{1,2,\ldots,n\}$ and let $S_i$ be segments of $\Ff$, $i=1,\ldots,n$. Then $f$ is a \emph{nested canalizing function (NCF)} in the variable order $x_{\sigma(1)},\ldots,x_{\sigma(n)}$ with canalizing input sets $S_{1},\ldots,S_{n}$ and canalized output values $b_{1},\ldots,b_{n},b_{n+1}$ with $b_n\neq b_{n+1}$ if it can be represented in the form
\[f(x_{1},\ldots,x_{n})=
\left\{\begin{array}[c]{ll}
b_{1} & x_{\sigma(1)}\in S_{1},\\
b_{2} & x_{\sigma(1)} \notin{ S_{1}}, x_{\sigma(2)}\in S_{2},\\
b_{3} & x_{\sigma(1)}\notin{ S_{1}}, x_{\sigma(2)} \notin{ S_{2}}, x_{\sigma(3)}\in S_{3},\\
\vdots  & \\
b_{n} & x_{\sigma(1)} \notin{ S_{1}},\ldots,x_{\sigma(n-1)} \notin{ S_{n-1}}, x_{\sigma(n)}\in S_{n},\\
{b_{n+1}} & x_{\sigma(1)} \notin{ S_{1}},\ldots,x_{\sigma(n-1)} \notin{ S_{n-1}}, x_{\sigma(n)}\notin{S_{n}}.
\end{array}\right.\]

In short, the function $f$ is said to be nested canalizing if $f$ is nested canalizing in some variable order with some canalizing input sets and some canalized output values.
\end{defn}

Let $\mathbb{S}=(S_{1},S_{2},\ldots,S_{n})$ and $\beta=(b_{1},b_{2},\ldots,b_{n+1})$ with $b_n\neq b_{n+1}$. We say that $f$ is $\{\sigma:\mathbb{S}:\beta\}$ NCF if it is nested canalizing in the variable order $x_{\sigma(1)},\ldots,x_{\sigma(n)}$, 
with canalizing input sets $\mathbb{S}=(S_{1},\ldots,S_{n})$ and canalized output values $\beta=(b_{1},\ldots, b_{n+1})$.

This definition immediately implies the following technical result, used subsequently. 

\begin{prop}\label{prop2.1}
A function $f$ is $\{\sigma:\mathbb{S}:\beta\}$ NCF  if and only if  $f$ is
$\{\sigma:\mathbb{S'}:\beta'\}$ NCF, where $\mathbb{S'}=(S_{1},S_{2},\ldots,{S_n}^c)$ and $\beta'=(b_{1},b_{2},\ldots,b_{n-2},b_{n+1},b_n)$.
\end{prop}
%
%


\section{Normalized average $c$-sensitivities}\label{5}
Gene regulatory networks must be robust to small perturbations in order to cope with the ubiquitous changes in environmental conditions. The Derrida plot is a common technique to quantify the robustness of a discrete dynamical system. It describes how a network perturbation of a certain size propagates on average over time~\cite{Derrida1}. A system operates in the ordered regime if a small perturbation vanishes on average over time. Such networks typically possess many steady states and short limit cycles. A system in which a small perturbation typically amplifies over time is in the chaotic regime, often possessing long limit cycles. Lastly, if a small perturbation remains on average of similar size, the system operates close to the so-called critical threshold. Working at this ``edge of chaos'' seems essential for many biological systems; it provides robustness to withstand perturbations caused by environmental changes but also flexibility to allow adaptation~\cite{Balleza, Nykter}.

In \cite{Kadelka_Derrida16}, the $c$-sensitivity of a function was introduced as a generalization of the popular concept of sensitivity. It measures how likely a simultaneous change in $c$ inputs of a function leads to a change in the output, and can be directly adopted to the multistate case.

\begin{defn}\label{def_c_sens_multi}
Any vector that differs at exactly $c$ bits from a given vector $\mathbf x$ is called a $c$-Hamming neighbor of $\mathbf x$. Let $f: \Ff^n \rightarrow \Ff$. The $c$-sensitivity of $f$ on $\mathbf x$ is defined as the number of $c$-Hamming neighbors of $\mathbf x$ on which the function value is different from its value on $\mathbf x$. That is,
$$S^f_c(\mathbf x) = \sum_{\substack{I \subseteq \{1,2,\ldots,n\}\\ | I | = c}} \chi[f(\mathbf x) \neq f(\mathbf x \oplus e_I)],$$
where $\chi$ is an indicator function, $\oplus$ is addition modulo $2$ and $e_I$ is a vector with nonzero entries at all indices in $I$ and $0$ everywhere else. Assuming a uniform distribution of $\mathbf x$,
$$S^f_c = \mathbb{E}[S^f_c(\mathbf x)] = \frac 1{p^n}  \sum_{\mathbf x \in \Ff^n} \sum_{\substack{I \subseteq \{1,2,\ldots,n\}\\ | I | = c}}\chi[f(\mathbf x) \neq f(\mathbf x \oplus e_I)]$$ 
is the average $c$-sensitivity of $f$. The range of $S^f_c$ is $[0,\binom nc]$. Let us therefore define the normalized average $c$-sensitivity of $f$ as $$q^f_c = \frac{S^f}{\binom nc} \in [0,1].$$
\end{defn}

It was also shown in \cite{Kadelka_Derrida16} that the Derrida values of a network can be expressed as a weighted sum of the normalized average $c$-sensitivities of its canalizing update functions. Since the Hamming weight of two multistate vectors is still defined as the number of bits where the vectors differ, the Derrida values of a network governed by multistate functions are also defined like in the Boolean case \cite{Kadelka_Derrida16}. The following theorem provides a formula for the normalized average $c$-sensitivities of multistate NCFs, which enables the calculation of the Derrida plot for any network governed by multistate NCFs - a task that otherwise requires extensive simulations, which is difficult or infeasible for large networks.

\begin{theorem}\label{thm_derrida}
For $p\geq 2$, the normalized average $c$-sensitivitiy of a multistate NCF $f: \Ff_p^n \rightarrow \Ff_p$ is
\begin{align*}
q_c^f &= \frac{p+1}{3p} \left( \frac{c2^c}{n2^n}\left(\frac{p}{p-1}-{}_2F_1\Big[1,n;n+1-c;\frac 12\Big]\right) \left(\frac{p-2}{3p-3}\right)^{c-1} \right. \\
& \qquad \qquad \left. + \sum_{i=1}^c \frac{\binom{n-i}{c-i}}{\binom{n}{c}} {}_2F_1\Big[i,c-n;i-n;\frac 12\Big] \left(\frac{p-2}{3p-3}\right)^{i-1}   \right),
\end{align*}
with ${}_2F_1$ the hypergeometric function.
\end{theorem}

\begin{proof}
See Appendix.
\end{proof}

For $p=2$, this formula reduces to the same simple equation as in \cite[Corrollary 3.4]{Kadelka_Derrida16}.



\section{Characterization of nested canalizing functions}\label{3}
As mentioned in the introduction, it is important to better understand the class of multistate NCFs, in particular
their density among all multistate functions. These questions have been largely addressed in the Boolean case,
but little information is available in the more general case. The main results in this section include a canonical parametrized representation of NCFs as a particular form of polynomial function. 
This representation is then used to derive a closed formula for the number of multistate NCFs in a given number of variables, as well as an asymptotic formula as the number of variables grows.
An important practical application of the polynomial form of NCFs is that it allows the easy generation of such functions by choosing a particular collection of parameters.
This is very helpful in simulation studies involving large numbers of networks governed by NCFs.
Finally, we derive a formula for the number of equivalence classes of NCFs under permutation of variables, a question that has received recent interest \cite{Reichhardt,Rocha}. 

\medskip\noindent
In the Boolean case, the extended monomial plays an important role in deriving a polynomial form of NCFs \cite{Yua1}. In the multistate case, the product of indicator functions, also used in \cite{Mur2}, takes over this role. 

\begin{defn}\label{def3.1}
Given a proper subset $S$ of $\Ff$, the indicator function (of $S^c$) is defined as
\[Q_S(x)=\left\{\begin{array}[c]{ll}
0 & x\in S,\\
1 & x\in S^c.\end{array}\right.\]
\end{defn}

The following theorem gives an algebraic characterization of NCFs.

\begin{thm}\label{th3.1}
For $n\geq2$, the function $f(x_{1},\ldots,x_{n})$ is nested canalizing if and only if  it can be uniquely written as
\begin{equation}\label{eq3.1}
f(x_{1},\ldots,x_{n})=M_{1}\Big(M_{2}\Big(\cdots\big(M_{r-1}(B_{r+1}M_{r}+B_r )+B_{r-1}\big)\cdots\Big)+B_2\Big)+B_1,
\end{equation}
where each $M_{i}$ is a product of indicator functions of disjoint sets of variables. More precisely, 

\begin{itemize}
\item $k_{1} + \cdots + k_{r}=n$, and $k_{i}\geq1$ for all $i=1,\ldots,r$
\item For all $i=1,\ldots,r$, $M_{i}=\prod_{j \in A_i}(Q_{S_j}(x_j))$, where $A_i \subseteq \{1,\ldots,n\}, |A_i| = k_i, A_{i_1} \cap A_{i_2} = \emptyset$ if $i_1 \neq i_2$, $A_1 \mathbin{\dot{\cup}} \ldots \mathbin{\dot{\cup}} A_r = \{1, \ldots, n\}$, and $S_1, \ldots, S_n$ are segments of $\Ff$
\item $B_1\in \Ff, B_2,\ldots,B_{r+1} \in \Ff - \{0\}$
\item if $k_r=1$, then $B_{r+1}+B_r\neq 0$
\end{itemize}

\end{thm}

\begin{proof}
First, let $b_i = \sum_{j=1}^i B_j$. Then it is straightforward to check that any function written as in Equation \ref{eq3.1} 
is a $\{\sigma':\mathbb{S'}:\beta'\}$ NCF, where
\begin{align*}
\sigma'(x_1,\ldots,x_n)&=(x_{1_1},\ldots,x_{1_{k_1}},\ldots,x_{r_1},\ldots,x_{r_{k_r}}),\\
\mathbb{S'}&=(S_{1_1},\ldots,S_{1_{k_1}},\ldots,S_{r_1},\ldots,S_{r_{k_r}}), \\
\beta'&=(\underbrace{b_1,\ldots,b_1}_{k_1},\underbrace{b_2,\ldots,b_2}_{k_2},\ldots,\underbrace{b_r,\ldots,b_r}_{k_r},b_{r+1}).
\end{align*}

Second, suppose $f$ is a $\{\sigma:\mathbb{S}:\beta\}$ NCF, where  $\mathbb{S}=(S_{1},S_{2},\ldots,S_{n})$ and $\beta=(b_{1},b_{2},\ldots,b_{n+1})$, with $b_n\neq b_{n+1}$. Then there exist $k_i,i=1,\ldots,r$, $k_1+\cdots+k_r=n$, $k_i\geq 1$, such that 
\begin{align*}
b_1=\cdots=b_{k_1}&=:C_1,\\
b_{k_1+1}=\cdots=b_{k_1+k_2}&=:C_2\\
&\ \vdots \\
b_{k_1+\cdots+k_{r-1}+1}=\cdots=b_n&=:C_r,\\
b_{n+1}&=:C_{r+1},\\ 
\text{and}\ C_j &\neq C_{j+1} \ \text{for all}\ j=1,\ldots, r
\end{align*}
Let $B_1:=C_1, B_2:=C_2-C_1, \ldots, B_{r+1}=C_{r+1}-C_r$. Hence, $B_1\in \Ff, B_2,\ldots,B_{r+1} \in \Ff - \{0\}$, and $f(x) = M_{1}(M_{2}(\cdots(M_{r-1}(B_{r+1}M_{r}+B_r )+B_{r-1})\cdots)+B_2)+B_1$, which shows that any NCF can be written as in Equation \ref{eq3.1}.

Finally, we need to show that each NCF has a unique polynomial representation. 
Let $f$ be written as in Equation \ref{eq3.1}. Then all the variables $x_{\sigma(1)},\ldots,x_{\sigma(k_1)}$ of $M_1$ are canalizing variables of $f$ with common canalized output $B_1$. To prove the uniqueness of $M_1$ and $B_1$, we will now show that $f$ has no other canalizing variables.
All variables of $M_2$, $x_{\sigma(k_1+1)},\ldots,x_{\sigma(k_2)}$, are canalizing variables of the subfunction $f_1:=M_{2}(\cdots(M_{r-1}(B_{r+1}M_{r}+B_r )+B_{r-1})\cdots)+(B_2+B_1)$. Since $B_1\neq B_1+B_2$, $x_{\sigma(k_1+1)},\ldots,x_{\sigma(k_2)}$ are not canalizing variables of $f$. In the same manner, all variables of $M_3$ are not canalizing variables of $f_1$ and thus not canalizing variables of $f$ either. 
Iteratively, we can prove that $x_{\sigma(1)},\ldots,x_{\sigma(k_1)}$ are the only canalizing variables of $f$, which proves the uniqueness of $M_1$ and $B_1$. In the same way, the uniqueness of $M_2, \ldots, M_r$ and $B_2, \ldots, B_{r+1}$ follows.
\end{proof}

\begin{cor}
Any choice of $r, k_i, A_i, B_j$ as in Theorem \ref{th3.1} results in an NCF. Thus, this formula can be used to generate NCFs with desired properties.
\end{cor}

\begin{remark}\label{remark2}
In the Boolean case, each $M_i$ in Theorem \ref{th3.1} is an extended monomial, and, since $B_{r+1}+B_r=1+1=0$, $k_r$ is greater than 1. Thus, Theorem \ref{th3.1} reduces to its Boolean version, already stated as Theorem 4.2 in \cite{Yua1}.
\end{remark}

Because each NCF can be uniquely written in the form of Equation \ref{eq3.1}, 
the number $r$ is uniquely determined by $f$, and can be used to specify the class of NCFs as in the Boolean case \cite{Yua1}. 
We can therefore use the structure of this polynomial form to define an additional structure of NCFs, which might shed light
on the dynamic behavior of NCF-governed networks. 

\begin{defn}\label{def3.3}
For an NCF $f$, written in the form of Equation \ref{eq3.1}, let the number $r$ be called its \emph{layer number}. 
Essential variables of $M_{1}$ are called most dominant variables (canalizing variables), and are part of the first layer of $f$. Essential variables of $M_{2}$ are called second most dominant variables, and are part of the second layer, etc.
\end{defn}

\begin{remark}\label{remark3}
Just like in the Boolean case, Equation~\ref{eq3.1} allows the use of Corollary 4.8 in \cite{Yua1}: The layer number of any NCF can be determined by counting the number of changes in the canalized output values. For example, if $p=3$ and if $f$ is nested canalizing with canalized output values $\beta = (1,0,2,2,0,1)$ ($n=5$), then the layer number of $f$ is $4$.
\end{remark}

We now derive some technical results needed in the construction of a closed formula for the number of NCFs. 

\begin{lemma}\label{lm3.1}
Let $a, b$ be any nonzero elements of $\Ff$, and let $S$ be any segment of $\Ff$. The number of different functions $f=bQ_S(x)+a$, which cannot be written as $cQ_{S'}(x)$, where $c\neq 0$ and $S'$ is a segment of $\Ff$, is
$(p-1)^2(p-2)$.
\end{lemma}

\begin{proof}
See Appendix.
\end{proof}

\begin{lemma}\label{lm3.2}
Given $a,b \neq 0$ and segments $S_i$, $i=1,\ldots,k$ with $k\geq 2$, then
\begin{enumerate}
\item $f(\xx)=f(x_1,\ldots,x_k)=b\prod_{j=1}^k Q_{S_j}(x_j)+a$ cannot be written as $c\prod_{j=1}^k Q_{S_j'}(x_j)$, where $c\neq 0$ and all $S_j'$ are segments, $j=1,\ldots, k$.
\item There are $2^{k}(p-1)^{k+2}$ different functions of the form $b\prod_{j=1}^k Q_{S_j}(x_j)+a$.
\end{enumerate}
\end{lemma}

\begin{proof}
See Appendix.
\end{proof}

Let $\mathbb{NCF}(n)$ denote the set of all NCFs in $n$ variables.

\begin{theorem}\label{thm3.2}
For $n\geq2$, the number of NCFs is given by
\begin{equation}\label{eq3.2}
|\mathbb{NCF}(n)|={2^np(p-1)^{n}\sum_{r=1}^{n}(p-1)^{r}r!\left[S(n,r)-\frac{np}{2}S(n-1,r)\right]}
\end{equation}
in terms of the Stirling numbers of the second kind
\[S(n,r)=\frac{1}{r!}\sum_{t=0}^{r}(-1)^t\binom{r}{t}(r-t)^{n}=\frac{1}{r!}\sum_{\substack{k_{1}+\cdots+k_{r}=n\\k_{i}\geq1,i=1,\ldots,r}}\frac{n!}{k_{1}!k_{2}!\cdots k_{r}!} \]
\end{theorem}

\begin{proof}
If $r=1$, then $f=B_2M_1+B_1$. Similar to Lemma \ref{lm3.2}, the number of such functions is $(2(p-1))^n(p-1)p=2^n(p-1)^{n+1}p$, 
since $B_1 \in \Ff$ can be arbitrarily chosen, and $k_1=n$.\\

For $r>1$, Equation \ref{eq3.1} yields that for each choice of $k_{1},\ldots, k_{r}$, $k_{i}\geq1$, $i=1,\ldots,r$, there are $(2(p-1))^{k_{j}}\binom{n-k_{1}-\cdots-k_{j-1}}{k_{j}}$ ways to form $M_{j}$, $j=1,\ldots,r$. For those NCFs with $k_r=1$, by Lemma \ref{lm3.1}, there are $(p-1)^2(p-2)$ different functions of the form $B_{r+1}M_r+B_r$ with $B_r, B_{r+1}\neq 0$. For the remaining NCFs, i.e., those with $k_r>1$, Lemma \ref{lm3.2} yields that there are $(p-1)^2(2(p-1))^{k_r}$ ways to form $B_{r+1}M_r+B_r$, with $B_r, B_{r+1}\neq 0$.

Note that there are $p-1$ choices for each $B_i$, $2\leq r\leq B_{r-1}$, $p$ choices for $B_1$, and $2(p-1)$ choices for each canalizing input segment. Hence, the total number of NCFs with $r>1, k_r=1$, can be given by
{\small
\begin{align*}
N_1&=\sum_{r=2}^n\sum_{\substack{k_{1}+\cdots+k_{r-1}=n-1\\k_{i}\geq1,i=1,\ldots,r-1}}(2(p-1))^{k_{1}+\cdots+k_{r-1}}\binom{n}{k_{1}}\binom{n-k_{1}}{k_{2}}\cdots\binom{n-k_{1}-\cdots-k_{r-2}}{k_{r-1}}(p-1)^2(p-2)(p-1)^{r-2}p\\
&=2^{n-1}p(p-2)\sum_{r=2}^n\sum_{\substack{k_{1}+\cdots+k_{r-1}=n-1\\k_{i} \geq1,i=1,\ldots,r-1}}(p-1)^{n+r-1}\frac{n!}{(k_{1})!(n-k_{1})!}\frac{(n-k_{1})!}{(k_{2})!(n-k_{1}-k_{2})!}\frac{(n-k_{1}-\cdots-k_{r-2})!}{k_{r-1}!(n-k_{1}-\cdots-k_{r-1})!}\\
&=2^{n-1}p(p-2)\sum_{r=2}^n(p-1)^{n+r-1}\sum_{\substack{k_{1}+\cdots+k_{r-1}=n-1\\ k_{i}\geq1,i=1,\ldots,r-1}}\frac{n!}{k_{1}!k_{2}!\cdots k_{r-1}!} \\
&=2^np(p-1)^n\sum_{r=1}^{n-1}(p-1)^{r}\left(\frac{p-2}{2}\right)nr!S(n-1,r)
\end{align*}}
where the last step follows by shifting the index of the sum.
Similarly, the total number of NCFs with $r>1, k_r>1$ is

{\small \begin{align*}
N_2&=\sum_{r=2}^{n-1}\sum_{\substack{k_{1}+\cdots+k_{r}=n\\k_{i}\geq1,i=1,\ldots,r-1,k_r\geq 2}}(2(p-1))^{k_{1}+\cdots+k_{r}}\binom{n}{k_{1}}\binom{n-k_{1}}{k_{2}}\cdots\binom{n-k_{1}-\cdots-k_{r-1}}{k_{r}}(p-1)^2(p-1)^{r-2}p\\
&=2^{n}p\sum_{r=2}^{n-1}\sum_{\substack{k_{1}+\cdots+k_{r}=n\\k_{i}\geq1,i=1,\ldots,r-1,k_r\geq 2}}(p-1)^{n+r}\frac{n!}{(k_{1})!(n-k_{1})!}\frac{(n-k_{1})!}{(k_{2})!(n-k_{1}-k_{2})!}\cdots \frac{(n-k_{1}-\cdots-k_{r-1})!}{k_{r}!(n-k_{1}-\cdots-k_{r})!}\\
&=2^{n}p\sum_{r=2}^{n-1}(p-1)^{n+r}\sum_{\substack{k_{1}+\cdots+k_{r}=n\\k_{i}\geq1,i=1,\ldots,r-1,k_r\geq 2}}\frac{n!}{k_{1}!k_{2}!\cdots k_{r}!} \\
&=2^np(p-1)^n\sum_{r=2}^{n-1}(p-1)^{r}\left[r!S(n,r)-n(r-1)!S(n-1,r-1)\right]
\end{align*}}
We can extend the upper limit to $n$ and shift the index of the second sum.
\[N_2= 2^np(p-1)^n\left[\sum_{r=2}^{n}(p-1)^{r}r!S(n,r) - \sum_{r=1}^{n-1}(p-1)^{r+1}nr!S(n-1,r)\right]\]
The $r=1$ term was previously calculated and neatly corresponds to a $r=1$ term in the first sum, so, by
combining all three groups of NCFs, the total number of NCFs in $n$ variables is
\begin{align*}
|\mathbb{NCF}(n)|&=2^n(p-1)^{n+1}p+N_1+N_2\\
&={2^np(p-1)^n\sum_{r=1}^{n}(p-1)^{r} r! \left[S(n,r)-\frac{np}{2}S(n-1,r)\right].}
\end{align*}
\end{proof}

Note that for $p=2$, we get the same formula as in \cite{Yua1}. However, we are now also able to explicitly compute the number of multistate NCFs. For example, when $p=3$ and $n=2,3,4$, we get 192, 5568, 219468, respectively; when $p=5$ and $n=2,3,4$, we get 5120, 547840, 78561280, respectively. These results are consistent with those calculated recursively in \cite{Mur2}.

By expressing Equation \ref{eq3.2} recursively, we get

\begin{cor}
\label{co3.1} For the nonlinear recursive sequence
\[a_{2}=4(p-1)^4, a_{n}=\sum_{r=2}^{n-1}\binom{n}{r-1}2^{r-1}(p-1)^ra_{n-r+1}+2^{n-1}(p-1)^{n+1}(2+n(p-2)) , n\geq3\]
it holds that 
\[\big|\mathbb{NCF}(n)\big| = pa_n,\]
and the explicit solution for $a_n$ is given by
\begin{align*}
a_n&={2^n(p-1)^{n}\sum_{r=1}^{n}(p-1)^{r}r!\left[S(n,r)-\frac{np}{2}S(n-1,r)\right].}
\end{align*}
\end{cor}

{

We are now in a position to derive an asymptotic formula for the number of NCFs, from a generating function. 

\begin{cor}
The exponential generating function of the number of multistate NCFs is
\[G_p(s) = \sum_{n=2}^{\infty}\frac{|\mathbb{NCF}(n)|}{n!} s^n  = \frac{p-p^2(p-1)s}{p-(p-1)\rme^{2(p-1)s}} -p -p(p-1)(p-2)s
\] 
\end{cor}

\begin{proof}
To obtain results for sums of the type
\[A_n(z)=\sum_{k=0}^{n} k!S(n,k) z^k \]
we use the exponential generating function of $S(n,k)$ with fixed $k$
\[\sum_{n=0}^{\infty} \frac{S(n,k)}{n!} s^n = \frac{\left(\rme^{s}-1\right)^k}{k!}\]
so that the exponential generating function of $A_n(z)$ is
\begin{equation} \label{genbellegf}
\sum_{n=0}^{\infty} \frac{A_n(z)}{n!} s^n = \sum_{k=0}^{\infty} \left(\rme^{s}-1\right)^k z^k = \frac{1}{1+z-z\rme^{s}}
\end{equation}
The $z=1$ case corresponds to the ordered Bell numbers.   Comparing to Equation~\ref{eq3.2}
\[|\mathbb{NCF}(n)| = 2^np(p-1)^n\left[A_n(p-1)-\frac{np}{2}A_{n-1}(p-1)\right] \]
directly gives the generating function for NCFs of the form
\[ \frac{p-p^2(p-1)s}{p-(p-1)\rme^{2(p-1)s}}\]
and we remove the unwanted $n=0$ and $n=1$ terms.
\end{proof}

From the generating function, one can obtain the number of NCFs by Taylor expansion or by performing a contour integral
\[|\mathbb{NCF}(n)| = \frac{n!}{2\pi \rmi}\oint \frac{G_p(s)}{s^{n+1}}\rmd s \]
We can evaluate the integral analogously to the example of the ordered Bell numbers treated in \cite{wilf94}.  Essentially, we start to deform the contour around the real simple pole so that the residue there provides the leading order asymptotic approximation for the number of multistate NCFs for large $n$.  Successive subleading corrections then arise from the complex poles.

\begin{cor}
The number of NCFs is approximately given by
\begin{equation} \label{ncf_approx}
|\mathbb{NCF}(n)| \approx \left[1-\frac{p}{2}\ln\left(\frac{p}{p-1}\right)\right]  2^n(p-1)^n n! \left[\ln\left(\frac{p}{p-1}\right) \right]^{-(n+1)}. \end{equation}
\end{cor}
}

\begin{figure}
\begin{center}
\includegraphics[width=0.49\textwidth]{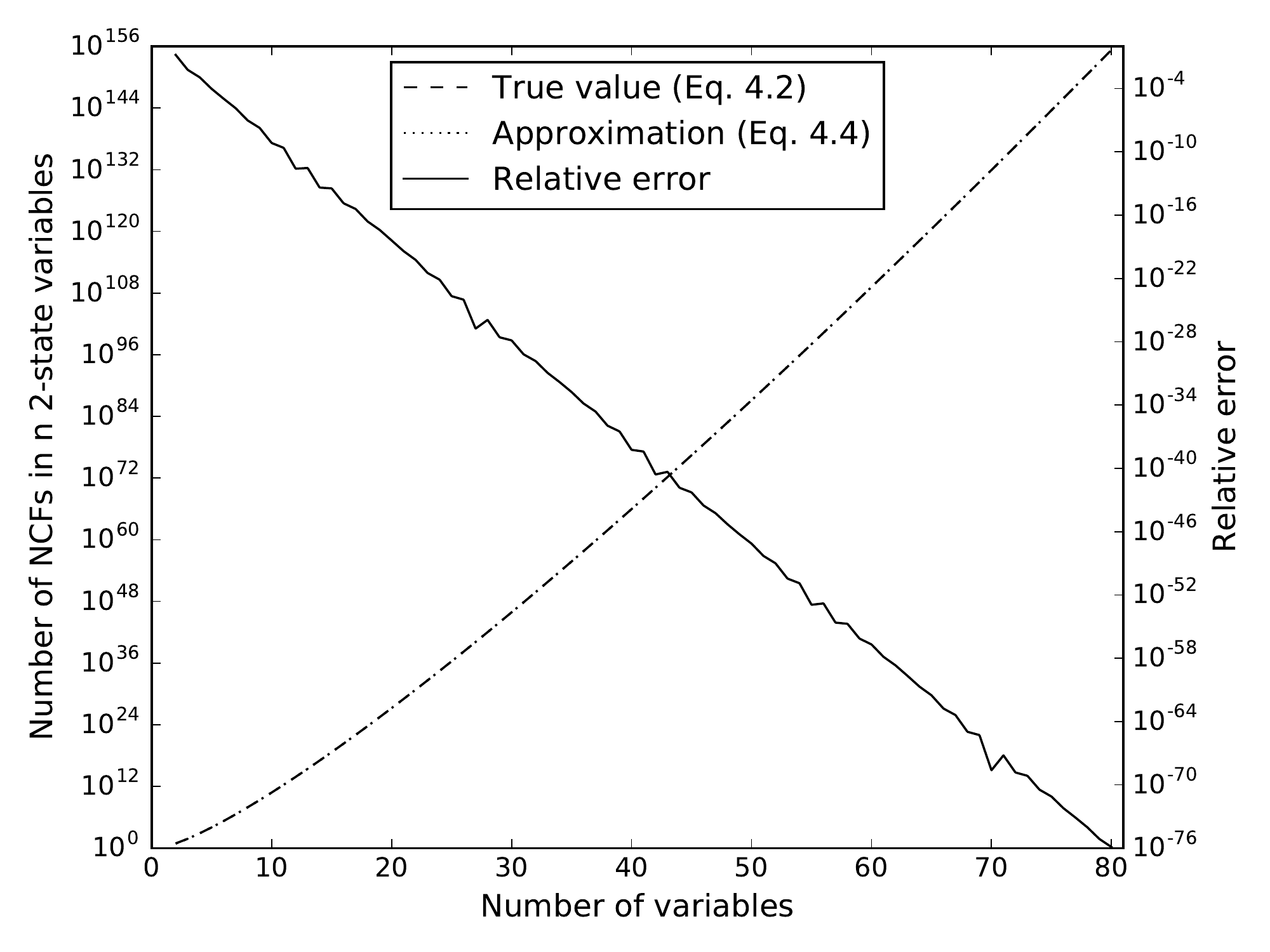}
\includegraphics[width=0.483\textwidth]{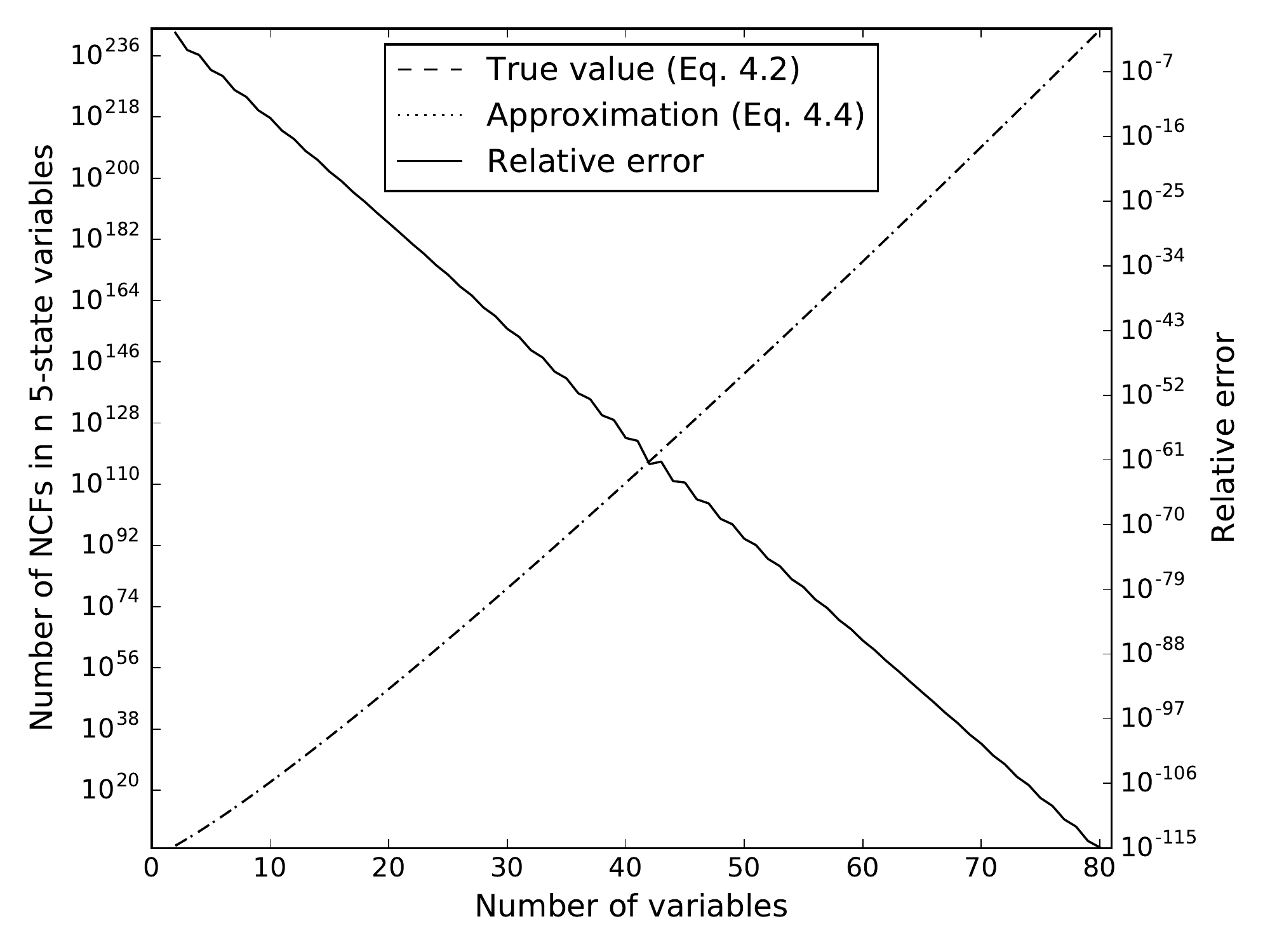}
\end{center}
\caption{The true number of NCFs (Equation \ref{eq3.2}) and the approximated number of NCFs (Equation \ref{ncf_approx}) as well as the relative error of the approximation are shown for $n=2,\ldots,80$ and $p=2$ (left panel), $p=5$ (right panel).}
\label{fig:approx}
\end{figure}

Figure \ref{fig:approx} shows the excellence of this approximation for $p=2$ and $p=5$.

We now study the number of equivalence classes of NCFs under permutation of variables. There has been recent interest in the study of various symmetries in the input variables of NCFs and the choice of representatives from different equivalence classes; see, e.g., \cite{Reichhardt,Rocha}.

\begin{defn}\label{def_equiv}
Given two functions $f(x_1,\ldots,x_n)$ and $g(x_1,\ldots,x_n)$ over $\Ff$. We call $f$ and $g$ permutation equivalent if there exists a permutation $\sigma$ such that $f(x_1,\ldots,x_n)=g(x_{\sigma(1)},\ldots,x_{\sigma(n)})$.
\end{defn}

Equivalent functions share many properties. For example, two equivalent Boolean NCFs have the same average sensitivity \cite{Yua1,Yua3} as well as the same average $c$-sensitivity (see Theorem \ref{thm_derrida}).

\begin{theorem}\label{th_equiv}
For $n\geq 2$, the number of different equivalence classes of NCFs under permutation of variables is
\[N=2^{n-1}(p-1)^{n+1}p^{n}.\]
\end{theorem}

\begin{proof}
See Appendix.
\end{proof}

The number of different equivalence classes of NCFs under permutation of variables is much lower than the number of NCFs. For example, when $p=3$ and $n=2,3,4$ we get 144, 1728, 20736 equivalence classes respectively, compared to the totals of 192, 5568, 219468 NCFs; when $p=5$ and $n=2,3,4$ we get 3200, 128000, 5120000 respectively, compared to 5120, 547840, 78561280.



\section{Discussion}\label{6}
For a class of functions to be a good representative of the mechanisms at work in gene regulation, it should be sufficiently large to capture many possible logic rules that appear, but should, at the same time, be small enough to endow networks with special properties that can be discerned. In our case, this necessitates a study of the class of general NCFs.
There are a few ways to generalize the concept of Boolean nested canalization. We treated a fairly restrictive but simple generalization, in which each variable can only appear once in the canalizing order. One limitation of this approach is that some Boolean NCFs are no longer nested canalizing when generalizing them to the multistate case. For instance, MIN and MAX, which can be seen as natural generalizations of the Boolean AND and OR functions, are not multistate NCFs. Another less restrictive but more complicated generalization would allow variables to appear in the canalizing order up to $p-1$ times, each time with a different canalizing input. This approach would, however, lead to uniqueness issues of the canonical parametric polynomial form. For instance, the layer number would no longer correspond to the number of changes in the canalized output vector.

While much work remains to be done in order to fully characterize the class of multistate NCFs, this paper provides a collection of tools for
a more in-depth study of systems governed by NCFs. The presented formula of the normalized average $c$-sensitivities of a multistate NCF allows the computation of Derrida values of NCF-governed networks, a commonly used metric of network stability. In the process, we have extended the definition of Derrida value from the Boolean to the multistate context. A very appealing closed form formula for the number of NCFs in a given number of variables is derived, in terms of the Sterling numbers of the second kind. 
A canonical parametric polynomial form of NCFs is derived that has important consequences: (i) it allows the easy construction of such functions, useful for simulation studies; (ii) it helps to derive an  asymptotic formula for the portion of NCFs among all functions; (iii) this polynomial form suggests a division of variables into layers which appear related to robustness properties captured by the Derrida values. Despite recent progress in \cite{Kadelka_Derrida16}, the precise connection remains to be elucidated. Our work builds in part on prior work done in the Boolean case \cite{Abd2, Mur, Yua1}.

As in \cite{Alan,Lau}, the results in this paper show the utility of the polynomial algebra viewpoint of discrete dynamical systems in general, and Boolean networks in particular, for the study of discrete dynamical systems in biology and elsewhere. Viewing polynomial algebra as a nonlinear version of linear algebra, the connection might not be surprising, and we believe that we have only scratched the surface in finding new applications of algebra to the study of nonlinear dynamical systems. 

\section*{Acknowledgements}
CK and RL were supported by NSF Grant CMMI-0908201 and US DoD Grant W911NF-14-1-0486. YL and JA were supported by US DoD Grant W911NF-11-10166.

\bibliographystyle{elsarticle-num}
\bibliography{ncf_bib_latest}

\section*{Appendix}\label{Appendix}
\subsection*{Theorem \ref{thm_derrida}}
\begin{proof}
Let $f$ be a $\{\sigma:\mathbb{S}:\beta\}$ NCF, as in Definition \ref{def2.3}, with $n$ essential variables. Let $\mathbf x=(x_1,\ldots,x_n), \mathbf y=(y_1,\ldots,y_n) \in \Ff_p^n$ be two system configurations that differ at $c$ of the $n$ positions. Let $\tilde{\sigma}$ be the restriction of $\sigma$ on the variables where $\mathbf x$ and $\mathbf y$ differ, $$\tilde{\sigma} = \{\sigma(i) | x_{\sigma(i)} \neq y_{\sigma(i)}\}.$$ The order of the elements in $\tilde\sigma$ is the same as in $\sigma$. 


(i) As an NCF, $f$ is evaluated in an iterative process. If in the evaluation of $f(\mathbf x)$ the most important value $x_{\sigma(1)}$ is in the canalizing set $S_1$, the evaluation is completed. If not, the second most important value, $x_{\sigma(2)}$, is considered, etc. In the Boolean case, whenever $\mathbf x$ and $\mathbf y$ differ at a variable $\sigma(i)$, $x_{\sigma(i)} \in S_i$ and $y_{\sigma(i)} \not\in S_i$ or $x_{\sigma(i)} \not\in S_i$ and $y_{\sigma(i)} \in S_i$. This implies that after evaluation of $\tilde\sigma(1)$, the most important variable where $\mathbf x$ and $\mathbf y$ differ, $\mathbf x$ and $\mathbf y$ always follow different paths in the evaluation process, and all less important variables do not matter when determining the probability that $f(\mathbf x)$ and $f(\mathbf y)$ differ. This is no longer true in the multistate case. For instance, if $p=3, S_1 = \{0\}, x_{\sigma(1)} = 1, y_{\sigma(1)} = 2$, then the second most important variable $\sigma(2)$ needs to be considered even though $\mathbf x$ and $\mathbf y$ differ at the most important variable $\sigma(1)$. Let 
$$\phi_1(p):= \mathbb{P}\Big([x_{\sigma(i)} \in S_i \wedge y_{\sigma(i)} \not\in S_i] \vee [x_{\sigma(i)} \not\in S_i \wedge y_{\sigma(i)} \in S_i] \Big| x_{\sigma(i)} \neq y_{\sigma(i)} \wedge S_i \in \mathcal{S}\Big),$$
where
$$\mathcal{S} = \big\{\{0\},\{0,1\},\ldots,\{0,\ldots,p-2\},\{p-1\},\{p-2,p-1\},\ldots,\{1,2,\ldots,p-1\}\big\}$$
is the set of all possible canalizing input segments. Clearly, $| \mathcal{S} | = 2(p-1)$, and there exist two segments with one element, two with two elements, etc. Moreover, there are $\binom p2$ pairs of $x_{\sigma(i)}$ and $y_{\sigma(i)}$ so that the two values are different. For a given $S_i$,
$$\mathbb{P}\Big([x_{\sigma(i)} \in S_i \wedge y_{\sigma(i)} \not\in S_i] \vee [x_{\sigma(i)} \not\in S_i \wedge y_{\sigma(i)} \in S_i] \Big| x_{\sigma(i)} \neq y_{\sigma(i)}\Big) = \frac{|S_i|(p-|S_i|)}{\binom p2}.$$
Thus,
\begin{align*}
\phi_1(p) =\sum_{S_i \in \mathcal{S}} \frac{|S_i|(p-|S_i|)}{2(p-1)\binom p2} = \frac{\sum_{i=1}^{p-1}i(p-i)}{(p-1)\binom p2} = \frac{p+1}{3(p-1)}.
\end{align*}

(ii) If the $j$th evaluation step is needed to determine the probability that $f(\mathbf x)$ and $f(\mathbf y)$ differ, and if $x_{\sigma(j)} \neq y_{\sigma(j)}$, then the probability that $\mathbf x$ and $\mathbf y$ follow different paths in the evaluation process at step $j$, i.e., the probability that no more steps are needed, is $\phi_1(p)$. If $j=n$ is the least important variable, then $f(\mathbf x)$ and $f(\mathbf y)$ differ for sure because $b_n \neq b_{n+1}$ in Definition~\ref{def2.3}. If $j<n$, then $b_j \neq b_q$ with probability $\frac{p-1}{p}$ for any $j<q<n$, so that
\begin{align*}
\phi_2(j,n,p) := \mathbb{P}\Big(f(\mathbf x) \neq f(\mathbf y) \Big| \big[ x_{\sigma(j)} \neq y_{\sigma(j)} \big] \wedge \big[ x_{\sigma(q)} \not\in S_q \wedge y_{\sigma(q)} \not\in S_q\ \forall &q, 1\leq q<j\big] \Big) =\\ 
&= \begin{cases}\phi_1(p)\frac{p-1}p & \ \text{if}\ j<n\\\phi_1(p) & \ \text{if}\ j=n\end{cases}
\end{align*}

(iii) When determining the probability that $f(\mathbf x)$ and $f(\mathbf y)$ differ, the $j$th evaluation step is only needed if $x_{\sigma(q)} \not\in S_q \ \text{and}\ y_{\sigma(q)} \not\in S_q$ for all $1\leq q < j$. For any $q$, 
$$\mathbb{P}\big(x_{\sigma(q)} \not\in S_q \wedge y_{\sigma(q)} \not\in S_q\big) = \begin{cases}\frac 12 & \ \text{if}\ x_{\sigma(q)} = y_{\sigma(q)}\\
\frac12 \big(1-\phi_1(p)\big) & \ \text{if}\ x_{\sigma(q)} \neq y_{\sigma(q)}\end{cases}$$
If $\sigma(j)$ is the $i$th most important variable where $\mathbf x$ and $\mathbf y$ differ (i.e., if $\sigma(j) = \tilde\sigma(i)$), then
\begin{align*}
\phi_3(i,j,p) := \mathbb{P}\Big(x_{\sigma(q)} \not\in S_q \wedge y_{\sigma(q)} \not\in S_q \ \forall q, 1\leq q < j\Big| \sigma(j) = \tilde\sigma(i) \Big) &= \prod_{q=1}^{j-1} \mathbb{P}\big(x_{\sigma(q)} \not\in S_q \wedge y_{\sigma(q)} \not\in S_q \big)\\
&= \left(\frac{1-\phi_1(p)}2\right)^{i-1} \left(\frac 12\right)^{j-i}
\end{align*}

(iv) The probability that the $j$th most important variable $\sigma(j)$ is the $i$th most important variable, when only considering those variables where $\mathbf x$ and $\mathbf y$ differ, is equal to the probability that a c-subset of $\{1,2,\ldots,n\}$ contains $j$ as its $i$th lowest element. There are $\binom nc$ c-subsets of $\{1,2,\ldots,n\}$. If $j$ is the $i$th lowest element, then there are $\binom{j-1}{i-1}$ choices for the $i-1$ lower elements and $\binom{n-j}{c-i}$ choices for the $c-i$ higher elements. Thus, for $1\leq i \leq c, i \leq j \leq n-c+i$,
$$\phi_4(i,j,c,n):=\mathbb{P}\big(\sigma(j) = \tilde\sigma(i)\big) = \frac{\binom{j-1}{i-1}\binom{n-j}{c-i}}{\binom nc}$$

(v) Only variables where $\mathbf x$ and $\mathbf y$ differ matter when deciding whether $f(\mathbf x)$ equals $f(\mathbf y)$. We therefore only consider these $c$ variables in the calculation of $q(c,n)$. The probability $\phi_4(i,j,c,n)$ describes how likely the $i$th most important variable where $\mathbf x$ and $\mathbf y$ differ, $\tilde\sigma(i)$, occurs at position $j$ when considering all variables. Therefore, $1\leq i \leq c$ and $i\leq j \leq n+i-c$. Thus,
\begin{align*}
q(c,n) &= \mathbb{P}\big(f(\mathbf x) \neq f(\mathbf y) | d(\mathbf x,\mathbf y) = c\big)\\
&= \sum_{i=1}^c \sum_{j=i}^{n+i-c} \mathbb{P}\Big(\sigma(j) = \tilde\sigma(i)\Big) \cdot \mathbb{P}\Big(x_{\sigma(q)} \not\in S_q \wedge y_{\sigma(q)} \not\in S_q \ \forall q, 1\leq q < j\Big| \sigma(j) = \tilde\sigma(i) \Big) \cdot \\
&\qquad\qquad\ \ \cdot \mathbb{P}\Big(f(\mathbf x) \neq f(\mathbf y) \big| \big[ x_{\sigma(j)} \neq y_{\sigma(j)} \big] \wedge \big[ x_{\sigma(q)} \not\in S_q \wedge y_{\sigma(q)} \not\in S_q\ \forall q, 1\leq q<j\big] \Big)\\
&= \sum_{i=1}^c \sum_{j=i}^{n+i-c} \phi_4(i,j,c,n) \phi_3(i,j,p) \phi_2(j,n,p)\\
&= {\phi_4(c,n,c,n) \phi_3(c,n,p) \phi_2(n,n,p)+\sum_{i=1}^c \sum_{j=i}^{\min(n+i-c,n-1)} \phi_4(i,j,c,n) \phi_3(i,j,p) \phi_2(j,n,p)}\\
&= {\phi_1(p)\left[\frac{c2^c}{n2^n} \left(\frac{1-\phi_1(p)}2\right)^{c-1}+\frac{p-1}{p}\sum_{i=1}^c \sum_{j=i}^{\min(n+i-c,n-1)} \phi_4(i,j,c,n) \phi_3(i,j,p)\right]}\\
&= {\phi_1(p)\left[\frac{c2^c}{n2^n} \left(\frac{1-\phi_1(p)}2\right)^{c-1}+\frac{p-1}{p}\binom{n}{c}^{-1}\sum_{i=1}^c s_i(c,n) \left(\frac{1-\phi_1(p)}2\right)^{i-1} \right]},
\end{align*}
where 
$$s_i(c,n) = \sum_{j=i}^{\min(i+n-c,n-1)} \binom{n-j}{c-i} \binom{j-1}{i-1} \left(\frac{1}{2}\right)^{j-i}, \ 1\leq i \leq c$$
can be expressed in terms of hypergeometric functions
$$s_i(c,n) = \binom{n-i}{c-i} {}_2F_1\Big[i,c-n;i-n;\frac 12\Big] -\delta_{c,i}\binom{n-1}{c-1} {}_2F_1\Big[1,n;n+1-c;\frac 12\Big] .$$
\end{proof}

\subsection*{Lemma \ref{lm3.1}}
\begin{proof}
If a function $f(x)=bQ_S(x)+a$ can be written as $cQ_{S'}(x)$, then
\[f=bQ_S(x)+a=\left\{\begin{array}[c]{ll}%
a & x\in S\\
a+b & x\in S^c\end{array}\right.
=\left\{\begin{array}[c]{ll}%
0 & x\in S'\\
c & x\in {S'}^c\end{array}\right.
=cQ_{S'}(x).\]
Since $a$ and $c$ are nonzero, $a+b = 0 \Leftrightarrow a = -b$ must hold for such a function. 
Since $\Ff$ contains $p-1$ nonzero numbers, there are $p-1$ choices for $b$ and $p-2$ choices for $a$, to obtain a function that cannot be written as $cQ_{S'}(x)$. 
Moreover, there are $2(p-1)$ different segments $S$, but only half of them lead to a different function
since every function can be expressed in two different ways:
\[bQ_S(x)+a=b(1-Q_{S^c}(x))+a=-bQ_{S^c}(x)+(a+b).\]
Thus, there are $(p-1)^2(p-2)$ different functions $f=bQ_S(x)+a$ that cannot be written as $cQ_{S'}(x)$.
\end{proof}

\subsection*{Lemma \ref{lm3.2}}
\begin{proof}
(1) Assume a function $f(\xx)=b\prod_{j=1}^k Q_{S_j}(x_j)+a$ can be written as $c\prod_{j=1}^k Q_{S_j'}(x_j)$, then
\begin{align*}
b\prod_{j=1}^k Q_{S_j}(x_j)+a &= \left\{\begin{array}[c]{ll}%
a & \exists j: x_j\in S_j\\
a+b & \forall j: x_j \in S_j^c \end{array}\right. \\
&= \left\{\begin{array}[c]{ll}%
a & \xx \in (S_1 \times \Ff^{k-1}) \cup \ldots \cup (\Ff^{k-1} \times S_k)\\
a+b & \xx \in S_1^c \times \ldots \times S_k^c \end{array}\right. \\
&=\left\{\begin{array}[c]{ll}%
0 \quad \ \ \ & \xx \in ({S_1'} \times \Ff^{k-1}) \cup \ldots \cup (\Ff^{k-1} \times {S_k'})\\
c & \xx \in {S_1'}^c \times \ldots \times {S_k'}^c \end{array}\right.\\
&=c\prod_{j=1}^k Q_{S_j'}(x_j).
\end{align*}
Since $a,c\neq 0$, $a+b=0$ must hold. Hence, 
\[S_1^c \times \cdots \times S_k^c = ({S_1'} \times \Ff^{k-1}) \cup \cdots \cup (\Ff^{k-1} \times {S_k'}).\]
This last statement is however impossible. Thus, there is no function $f(\xx)=b\prod_{j=1}^k Q_{S_j}(x_j)+a$ that can be written as $c\prod_{j=1}^k Q_{S_j'}(x_j)$.

(2) The nonzero constants $a,b$ can be arbitrarily chosen, with $p-1$ choices each. Contrary to the previous lemma, each choice of segments $S_1, \ldots, S_k$ leads to a different function because $S_1^c \times \ldots \times S_k^c \neq ({S_1} \times \Ff^{k-1}) \cup \ldots \cup (\Ff^{k-1} \times {S_k})$ and because $a \neq a+b$. For each segment there are $2(p-1)$ choices, so that altogether there are $(p-1)(p-1)(2(p-1))^k = 2^k(p-1)^{k+2}$ different functions of the form $b\prod_{j=1}^k Q_{S_j}(x_j)+a$.
\end{proof}

\subsection*{Theorem \ref{th_equiv}}
\begin{proof}
In order to find the number of different equivalence classes of NCFs, we need the following combinatorial result \cite[Page 70]{Cha}:

Given $n$, $r$ and $s_i$, $i=1,\ldots,r$ and $s=s_1+\cdots+s_r\leq n$. 
Then the number of integer solution of the equation $k_1+\cdots+k_r=n$, where $k_i\geq s_i$, is
\[\sum_{\substack{k_{1}+\cdots+k_{r}=n\\k_{i}\geq s_i, s=s_1+\cdots +s_r\leq n}}1 = \binom{r+n-s-1}{r-1}.\]

The number of different equivalent classes of NCFs equals the number of different NCFs with a fixed canalizing variable order $\sigma$ in Equation \ref{eq3.1}. Thus, we can follow the same enumerative approach as we did in the proof of Theorem \ref{thm3.2}, except that we do not consider the permutation of the variables. Hence, 

\begin{align*}
N&=2^{n-1}p(p-2) \sum_{\substack{r=2}}^{n}(p-1)^{n+r-1}\sum_{\substack{k_{1}+\cdots+k_{r-1}=n-1\\k_{i}\geq1,i=1,\ldots,r-1 }}1\\
&\qquad +2^{n}p\sum_{\substack{r=1}}^{n-1}(p-1)^{n+r}\sum_{\substack{k_{1}+\cdots+k_{r}=n\\k_{i}\geq1,i=1,\ldots,r-1,k_{r}\geq2}}1\\
&=2^{n-1}p(p-1)^{n+1} \Big[(p-2)\sum_{\substack{r=2}}^{n}(p-1)^{r-2}\binom{n-2}{r-2}+2\sum_{\substack{r=1}}^{n-1}(p-1)^{r-1}\binom{n-2}{r-1}\Big]\\
&=2^{n-1}p(p-1)^{n+1} \Big[(p-2)p^{n-2} + 2p^{n-2}\Big]\\
&=2^{n-1}(p-1)^{n+1}p^{n},
\end{align*}
where we used the above mentioned lemma to eradicate the inner sums in the first equality and the binomial theorem to simplify the sums in the second equality.
\end{proof}

\end{document}